\documentclass[12pt, reqno]{amsart}
\usepackage{amsmath, amsthm, amscd, amsfonts, amssymb, graphicx, color}
\usepackage[bookmarksnumbered, colorlinks, plainpages]{hyperref}
\input{mathrsfs.sty}
\newtheorem{theorem}{Theorem}[section]
\newtheorem{lemma}[theorem]{Lemma}
\newtheorem{proposition}[theorem]{Proposition}
\newtheorem{corollary}[theorem]{Corollary}
\theoremstyle{definition}
\newtheorem{definition}[theorem]{Definition}

\theoremstyle{remark}
\newtheorem{remark}[theorem]{Remark}
\numberwithin{equation}{section}

\newcommand{\A}{{\mathscr A}}
\newcommand{\B}{{\mathscr B}}
\newcommand{\E}{{\mathscr E}}
\newcommand{\F}{{\mathscr F}}

\begin{document}
\title[Conditionally positive definite kernels in Hilbert $C^*$-modules]{Conditionally positive definite kernels in Hilbert $C^*$-modules}
\author[M.S. Moslehian ]{Mohammad Sal Moslehian}
\address{Department of Pure Mathematics, Ferdowsi University of Mashhad, P. O. Box 1159, Mashhad 91775, Iran}
\email{moslehian@um.ac.ir}
\begin{abstract}
We investigate the notion of conditionally positive definite in the context of Hilbert $C^*$-modules and present a characterization of the conditionally positive definiteness in terms of the usual positive definiteness. We give a Kolmogorov type representation of conditionally positive definite kernels in Hilbert $C^*$-modules. As a consequence, we show that a $C^*$-metric space $(S, d)$ is $C^*$-isometric to a subset of a Hilbert $C^*$-module if and only if $K(s,t)=-d(s,t)^2$ is a conditionally positive definite kernel. We also present a characterization of the order $K'\leq K$ between conditionally positive definite kernels.
\end{abstract} \maketitle

\section{Introduction}

Let $\mathbb{B}(\mathcal{H})$ denote the $C^*$-algebra of all bounded linear operators on a complex Hilbert space $(\mathcal H, \langle\cdot,\cdot\rangle)$. According to the Gelfand--Naimark--Segal theorem, every $C^*$-algebra can be regarded as a $C^*$-subalgebra of $\mathbb{B}(\mathcal{H})$ for some Hilbert space $\mathcal{H}$. An operator $A \in \mathbb{B}(\mathcal{H})$ is called \emph{positive} if $\langle Ax, x\rangle\geq 0 $ for all $x \in \mathcal{H}$, and we then write $A\geq 0$. In the case when the dimension of $\mathcal{H}$ is finite, that is when we deal with matrices, it is custom to use the terminology `\emph{positive semi-definite}'. To unify our approach, we however use the word `positive' instead of `positive semi-definite'.

A matrix $A=[a_{ij}]$ in $\mathbb{M}_n(\A)$, the $C^*$-algebra  of $n \times n$ matrices with entries in $\A$, is positive if and only if $\sum_{i,j=1}^na_i^*a_{ij}a_j \geq0$ for all $a_1, \cdots, a_n \in\A$. It follows from \cite[Lemma IV.3.2]{TAK} that a matrix in $\mathbb{M}_n(\A)$ is positive if and only if it is a sum of matrices of the form $[c_i^*c_j]$ for some $c_1, \cdots, c_n \in \A$.

The notion of \emph{semi-inner product $C^*$-module (Hilbert $C^*$-module}, resp.) is a natural generalization of that of semi-inner produce space (Hilbert space, resp.) arising by replacing the field of scalars $\mathbb{C}$ by a $C^*$-algebra. Recall that if $(\E,\langle \cdot ,\cdot\rangle) $ is a Hilbert $C^*$-module over a $C^*$-algebra $\A$, then for every $x\in \E$ the norm on $\E$ is given by $\|x\|=\|\langle x,x\rangle\|^{\frac{1}{2}}$ and the the ``absolute-value norm'' is defined by $|x|=\langle x,x \rangle ^{\frac{1}{2}}$ as a positive element of $\A$. A map $T:\E\longrightarrow \F$ between Hilbert $C^*$-modules is adjointable if there is a map $T^*:\F\longrightarrow \E$ such that $\langle Tx,y\rangle=\langle x,T^*y \rangle$ for all $x \in\E$ and $y \in \F$. Then $T$ (and $T^*$) are bounded $A$-linear maps by the uniform boundedness theorem. The set of all adjointable maps $T:\E\longrightarrow\F$ is denoted by $\mathcal{L}(\E,\F)$. It is known that $\mathcal{L}(\E):=\mathcal{L}(\E,\E)$ endowed with the operator norm is a unital $C^*$-algebra. A Hilbert $C^*$-module $\E$ is called self-dual if for every bounded $A$-linear map $f: \E \to \A$ there exist some $x_0\in\E$ such that $f(x)=\langle x_0,x\rangle$ for all $x\in \E$, see \cite[\S 2.5]{MT}.
A \emph{$C^*$-metric} on a set $S$ with values in a $C^*$-algebra $\A$ is a map $d$ from $S\times S$ into the cone $\A_+$ of positive elements of $A$ satisfying the same axioms as those of the usual metric when we consider the usual order $\leq$ on the real space of self-adjoint elements induced by the positive cone $\A_+$. See \cite{MA} for some examples of $C^*$-metrics. By a ``$C^*$-isometry'' $V$ from a $C^*$-metric space $S$ into a Hilbert $C^*$-module $\E$ we mean one satisfying $d(s,t)=|V(s)-V(t)|$.\\
The Cauchy--Schwarz inequality (see also \cite{ABM}) for $x,y$ in a semi-inner product $C^*$-module $\E$ asserts that
\begin{eqnarray}\label{CS}
\langle x,y\rangle \langle y,x\rangle\leq \|\langle y,y\rangle \| \langle x,x\rangle .
\end{eqnarray}
The reader is referred to \cite{FRA, LAN, MT} for more information on Hilbert $C^*$-modules.

It seems that positive definite kernels are first examined in 1904 by Hilbert, and conditionally positive definite kernels by Schoenberg around 1940 in a series of papers. Schoenberg also use conditionally positive definite kernels to embed a metric space into a Hilbert space. The (Kolmogorov) representation of positive definite kernels was first established by Kolmogorov in the scalar theory. The reader may consult \cite[pages 84-85]{BCR} for a complete historical view.  The Kolmogorov decomposition of positive definite kernels in context of Hilbert $C^*$-modules was given by Murphy \cite{MUR}. This decomposition asserts that any positive definite kernel $L: S\times S \to \mathcal{L}(\E)$ for a given Hilbert $C^*$-module $\E$ is of the form $L(s,t)=V(s)^*V(t)$, where $V$ is a map from $S$ into $\mathcal{L}(\E,\F)$ for some Hilbert $C^*$-module $\F$.

In this paper, we first state the notion of conditionally (called also almost) positive definite kernel in the context of Hilbert $C^*$-modules as a generalization of that in the scalar theory (cf. \cite{BCR, DON, EL}) and investigate some of its basic properties. Our investigation relies on the construction of Kolmogorov decomposition given by Murphy \cite{MUR}. Giving a characterization of conditionally positive definite kernels in Hilbert $C^*$-modules, we show that a $C^*$-metric space $(S, d)$ is $C^*$-isometric to a subset of a Hilbert $C^*$-module if and only if $K(s,t)=-d(s,t)^2$ is a conditionally positive definite kernel. We also present a characterization of the majorization $K'\leq K$ between conditionally positive definite kernels. Among other things, we present a Cauchy--Schwarz inequality for positive definite kernels with values in Hilbert $C^*$-modules.\\
Throughout the paper, $S$ stands for a non-empty set and $\A, \B$ denote $C^*$-algebras. Hilbert $C^*$-modules are denoted by $\E, \F$. For the sake of convenience, we usually use the letter $L$ for positive definite kernels and $K$ for conditionally positive definite kernels.


\section{Conditionally positive definite kernels and Conditionally positive matrices}

We model some techniques of the scalar theory of conditionally positive definite kernels to the context of Hilbert $C^*$-modules. We start our work with the following definition.

\begin{definition}
A hermitian matrix $A=[a_{ij}]\in \mathbb{M}_n(\A)\,\, (n\geq 2)$ is called \emph{conditionally positive} if $\sum_{i,j=1}^na_i^*a_{ij}a_j \geq0$ whenever $\sum_{i=1}^na_i=0$.
\end{definition}
Evidently, conditionally positive matrices form a positive cone containing positive matrices as well as matrices of the form $[c_i+c_j^*]$ for $c_1, \cdots, c_n \in \A$. This notion is closely related to some significant classes of real functions. For example, if $f: (0,\infty) \to (0,\infty)$ is an operator concave function, then all Loewner matrices associated with $f$ are conditionally positive; see \cite{BS}.
This notion has several application in harmonic analysis, physics and probability theory; see \cite{BCR} and references therein. There are other generalizations of this notion as well; see \cite[\S 8]{FM}.

By a kernel we mean any map on $S\times S$ into a $C^*$-algebra for some set $S$. For any kernel $L: S\times S \to \A$ one can define the kernel $L^*: S\times S \to \A$ by $L^*(s,t)=L(t,s)^*$. A kernel $L$ with values in a $C^*$-algebra is called hermitian if $L^*=L$.
\begin{definition}
A hermitian map $L:S \times S \to \A$ is called a \emph{positive definite kernel} if
\begin{eqnarray}\label{pdk1}
\sum_{i, j=1}^n a_i^*L(s_i,s_j)a_j \geq 0.
\end{eqnarray}
for every positive integer $n$, every $s_1, \cdots, s_n\in S$ and every $a_1, \cdots, a_n\in\A$.
\end{definition}
If $L$ is a kernel with self-adjoint values, then \eqref{pdk1} holds if $\sum_{i, j=1}^n a_iL(s_i,s_j)a_j \geq 0$ for any self-adjoint $a_1, \cdots,  a_n\in\A$. This easily follows by utilizing the Cartesian decomposition of any $a_i$ into its real and imaginary parts.
\begin{definition}
If a hermitian kernel $K:S \times S \to \A$ satisfies
\begin{eqnarray}\label{pdk2}
\sum_{i, j=1}^n a_i^*K(s_i,s_j)a_j \geq 0.
\end{eqnarray}
for all positive integers $n\geq 2$, all $s_1, \cdots, s_n\in S$ and all $a_1, \cdots, a_n\in\A$ subject to the condition $\sum_{i=1}^na_i=0$, then it is called \emph{conditionally positive definite}.
\end{definition}
It follows from inequalities \eqref{pdk1} and \eqref{pdk2} that the conditionally positive definiteness of $K$ and the positive definiteness of $L$ are equivalent to those of the matrices $[L(s_i,s_j)]$  and $[K(s_i,s_j)]$ in $\mathbb{M}_n(\A)$ for all $s_1, \cdots, s_n\in S$, respectively.

Clearly the set of all positive definite kernels and the the set of all conditionally positive definite kernels constitute positive cones. It is easy to check that any the so-called \emph{Gram kernel} $L(s,t)=g(s)^*g(t)$ and any kernel of the form $K(s,t)=g(s)+g(t)^*$, where $g: S\to \A$ is a map, are positive definite and conditionally positive definite, respectively. It immediately follows from the definition that if $L$ is a positive definite kernel, then $L(s,s)\geq 0$ and $L(t,s)=L(s,t)^*$ for all $s,t \in S$.

The next theorem is known in the literature for scalar kernels; cf. see \cite{BCR}. It was first proved in a special case by Schoenburg \cite{SCH}. We however states its proof in the noncommutative setting of $C^*$-algebras to show that how the conditionally positive definiteness differs from the positive definiteness.

\begin{theorem}\label{th1}
Let $K: S\times S \to \A$ be a hermitian kernel and $s_0\in S$. Then $K$ is conditionally positive definite if and only if the kernel $L: S\times S \to \A$ defined by
$$L(s,t):=\frac{1}{2}\left[K(s,t)-K(s,s_0)-K(s_0,t)+K(s_0,s_0)\right]$$ is a positive definite kernel.
\end{theorem}
\begin{proof}
Let $K$ be conditionally positive definite, $s_1, \cdots, s_n \in S$ and $a_1, \cdots, a_n\in A$. Set $a_{n+1}:=-\sum_{i=1}^na_i$ and $s_{n+1}:=s_0$. Then
\begin{eqnarray*}
0 &\leq& \frac{1}{2}\sum_{i,j=1}^{n+1}a_i^*K(s_i,s_j)a_j\\
&=& \frac{1}{2}\sum_{i,j=1}^{n}a_i^*K(s_i,s_j)a_j+\frac{1}{2}\sum_{i=1}^{n}a_i^*K(s_i,s_0)a_{n+1}\\
&&\quad +\frac{1}{2}\sum_{j=1}^{n}a_{n+1}^*K(s_0,s_j)a_j+\frac{1}{2}a_{n+1}^*K(s_0,s_0)a_{n+1}\\
&=&\frac{1}{2}\sum_{i,j=1}^{n}a_i^*\left[K(s_i,s_j)-K(s_i,s_0)-K(s_0,s_j)+K(s_0,s_0)\right]a_j\\
&=& \sum_{i,j=1}^na_i^*L(s_i,s_j)a_j\,
\end{eqnarray*}
whence we conclude that $L$ is positive definite.

Conversely, let $L$ be positive definite, $n\geq 2$, $s_1, \cdots, s_n \in S$ and $a_1, \cdots, a_n\in A$ with $\sum_{i=1}^na_i=0$. Then
\begin{eqnarray*}
0 &\leq& 2 \sum_{i,j=1}^na_i^*L(s_i,s_j)a_j\\
&=&\sum_{i,j=1}^na_i^*K(s_i,s_j)a_j-\left(\sum_{i=1}^na_i^*K(s_i,s_0)\right)\left(\sum_{j=1}^na_j\right)\\
&&\quad -\left(\sum_{i=1}^na_i^*\right)\left(\sum_{j=1}^nK(s_0,s_j)a_j\right)+\left(\sum_{i=1}^na_i^*\right)K(s_0,s_0)\left(\sum_{j=1}^na_j\right)\\
&=&\sum_{i,j=1}^na_i^*K(s_i,s_j)a_j\,.
\end{eqnarray*}
Hence $K$ is conditionally positive definite.
\end{proof}
\begin{remark}
Under the notation as in Theorem \ref{th1}, the positive definite kernel $L$ is identically zero if and only if there is a function $h:S \to\A$ such that $K(s,t)=h(s)+h(t)^*$. The sufficiency is clear. To prove the necessity, it is enough to put $h(s):=K(s,s_0)-\frac{1}{2}K(s_0,s_0)$.
\end{remark}
As a consequence of Theorem \ref{th1}, we state the following assertion for matrices.
\begin{corollary}\label{corm}
A matrix $[a_{ij}]$ in $\mathbb{M}_n(\A)$ is conditionally positive if and only if the matrix $[a_{ij}-a_{im}-a_{mj}+a_{mm}]$ is positive for some $1 \leq m\leq n$.
\end{corollary}


\begin{corollary}\label{cor1}
A self-adjoint $2 \times 2$ block matrix $\begin{bmatrix} A & B \\ B^* & C \end{bmatrix}$ of operators in $\mathbb{B}(\mathscr{H})$ is conditionally positive if and only if $A+C\geq B+B^*$.
\end{corollary}
\begin{proof}
Employing Corollary \ref{corm} with $m=2$, we see that $\begin{bmatrix} A & B \\ B^* & C \end{bmatrix}$ is conditionally positive if and only if $\begin{bmatrix} A+C-B-B^* & 0 \\ 0 & 0\end{bmatrix}$ is positive.
\end{proof}

\begin{remark}
A linear map $\Phi :\A\longrightarrow \mathcal{B}$ between $C^*$-algebras is said to be positive if $\Phi(A)\geq 0,$ whenever $A \geq 0.$ A linear map $\Phi$ is called \emph{$n$-positive} if the map $\Phi_n : M_n(\A)\longrightarrow M_n(\mathcal{B})$ defined by $\Phi_n([a_{ij}]) = [\Phi(a_{ij})]$ is positive. A map $\Phi$ is said to be \emph{completely positive} if it is $n$-positive for every $n\in \mathbb{N}$.

Given a linear map $\Phi$ between $C^*$-algebras, Corollary \ref{cor1} ensures that if $\Phi$ is positive, then the corresponding map $\Phi_2$ preserves the conditional positivity. The converse can be seen to be true by considering the conditionally positive matrix $\begin{bmatrix} A & 0 \\ 0 & 0\end{bmatrix}$. So it seems that to give a definition of completely conditionally positive map in the sense that it takes any conditionally positive matrix to a conditionally positive one needs some care. In this direction, one may say a map $\Phi: \A\ \to \B$ between $C^*$-algebras to be \emph{completely conditionally positive} if $\sum_{i,j=1}^na_i^*\Phi(b_i^*b_j)a_j \geq0$ for all $n\geq 2$, $b_1, \cdots, v_n\in \A$ and $a_1, \cdots, a_n\in \A$ with $\sum_{i=1}^na_ib_i=0$. Such maps have interesting properties in studying some types of semigroups; see \cite{BBLS}.
\end{remark}

From Corollary \ref{cor1} applied to the conditionally positive $2\times 2$ matrix $$\begin{bmatrix} K(s,s) & K(s,t) \\ K(t,s) & K(t,t) \end{bmatrix}$$ we derive that if $K$ is a conditionally positive definite kernel, then $$2{\rm Re} K(s,t) \leq K(s,s)+K(t,t)$$ for all $s, t\in S$.\\
To achieve a Cauchy--Schwarz inequality for conditionally positive definite kernels we need the following lemma.
\begin{lemma}\cite[Lemma 5.2]{LAN}\label{pos}
Let $\mathscr{H}$ be a Hilbert space and $T,P,Q \in \mathbb{B}(\mathscr{H})$ with $P\geq 0$ and $Q\geq 0$. If
$\begin{bmatrix} P & T \\T^* & Q\end{bmatrix}\geq 0$ in $\mathbb{B}(\mathscr{H}\oplus \mathscr{H})$, then $TT^*\leq \|Q\|P$.
\end{lemma}
Now, for a positive definite kernel $L$, it follows from the positivity of the matrix
$$\begin{bmatrix} L(s,s) & L(s,t) \\ L(t,s) & L(t,t) \end{bmatrix}$$
and Lemma \ref{pos} that the following Cauchy--Schwarz inequality holds
$$L(s,t)L(t,s) \leq \|L(t,t)\| L(s,s)\,.$$
This inequality is a generalization of \cite[Theorem 1.14]{EL} to positive definite kernels with values in $C^*$-algebras. We summarize the above facts as a proposition.
\begin{proposition} [Cauchy--Schwarz inequality]Let $S$ be a set.
\begin{itemize}
\item [(i)] Let $K$ be a conditionally positive definite kernel on $S$ with values in a $C^*$-algebra. Then $2{\rm Re} K(s,t) \leq K(s,s)+K(t,t)$ for all $s, t\in S$.
\item [(ii)] Let $L$ be a positive definite kernel on $S$ with values in a $C^*$-algebra. Then  $L(s,t)L(t,s) \leq \|L(t,t)\| L(s,s)$ for all $s, t\in S$.
\end{itemize}
\end{proposition}

\begin{remark}
Let $\A$ be a unital commutative $C^*$-algebra. By the Gelfand theorem it is of the form $C(\Omega)$ for some compact Hausdorff space $\Omega$. It is known that the Schur product $A\circ B=[a_{ij}b_{ij}]$ of two positive matrices $A=[a_{ij}]$ and $B=[b_{ij}]$ is positive. Hence the product $K_1\circ K_2: S\times S\to \A$ of two positive definite kernels defined by $(K_1\circ K_2)(s,t)=K_1(s,t)K_2(s,t)$ is again positive definite. This property is, however, not true for conditionally positive definite kernels. It is enough to consider $S=\{1,2\}$ and conditionally positive matrices $A=B=\begin{bmatrix} 0 & -1 \\ -1 & 0 \end{bmatrix}$ and use Corollary \ref{cor1}.
\end{remark}


\section{Kolmogorov decomposition of conditionally positive definite kernels}

Murphy \cite{MUR} established a Kolmogorov decomposition of positive definite kernels in context of Hilbert $C^*$-modules inspired by the scalar version in \cite{EL}. Utilizing his construction, we establish the Kolmogorov decomposition of conditionally positive definite kernels in the setting of Hilbert $C^*$-modules. Such constructions can be found in some types of Steispring theorems in various settings \cite{AMY, MJJ}.

\begin{theorem}[Kolmogorov decomposition]\label{th2}
Let $K$ be a conditionally positive definite kernel on a set $S$ into the $C^*$-algebra $\mathcal{L}(\E)$ for some Hilbert $\A$-module $\E$. Then there exist a Hilbert $C^*$-module $\F$ over $\A$ and a mapping $V: S \to \mathcal{L}(\E,\F)$ such that
\begin{eqnarray}\label{v}
K(s,t)=2V(s)^*V(t)-V(s)^*V(s)-V(t)^*V(t)-h(s)-h(t)^*,
\end{eqnarray}
where $h: S \to \mathcal{L}(\A)$ is a certain map.
\end{theorem}
\begin{proof}
Let us fix $s_0 \in S$ and set
\begin{eqnarray}\label{l}
L(s,t)=\frac{1}{2}\left(K(s,t)-K(s,s_0)-K(s_0,t)+K(s_0,s_0)\right)\,.
\end{eqnarray}
Employing Theorem \ref{th1} we deuce that $L$ is a positive definite kernel. Now we use the strategy in \cite{MUR} to construct the required Hilbert $C^*$-module $\F$ and the map $V$.

For any map $f:S \to \E$ with finite support we define $\L f: S \to \E$ by
$$\L f(s):=\sum_{t\in S}L(s,t)f(t)\,.$$
Denote by $\F_0$ the semi-inner product $\A$-module of all maps $\L f$ equipped with the pointwise operations and the semi-inner product
$$\langle \L f,\L g\rangle:=\sum_{s,t\in S}\langle L(s,t)f(t),g(s)\rangle\,.$$
A standard argument based on the Cauchy--Schwarz inequality \eqref{CS} shows that $\langle \cdot,\cdot\rangle$ is indeed an inner product. We denote the completion of $\F_0$ by $\F$, which is called the \emph{reproducing kernel space}. It is easy to verify that $\mathcal{L}(\E,\F)$ is a Hilbert $C^*$-module over $\mathcal{L}(E)$ via the inner product $\langle T,S\rangle:=T^*S$. \\ Next, we set $V: S \to \mathcal{L}(\E,\F)$ by $V(s)(x)=\L(x_s)$, where $x_s: S \to \E$ is defined by
$$x_s(t)=\Big\{\begin{matrix} 0 & t\neq s \\ x & t=s\end{matrix}.$$
It is not hard to see that $V$ is well-defined. In addition, $V(s)^*V(t)=L(s,t)$, since
$$\langle V(s)^*V(t)x,y\rangle=\langle \L x_t, \L y_s\rangle=\langle L(s,t)x,y\rangle$$
for all $x, y\in \E$. It is notable that $\cup_{s\in S}V(s)E$ is dense in $\F$.
We have
\begin{eqnarray*}
&&\hspace{-1.5cm} 2V(s)^*V(t)-V(s)^*V(s)-V(t)^*V(t)\\
&=&2\Big({\rm Re}(V(s)^*V(t))+2{\rm i}\, {\rm Im}(V(s)^*V(t))\Big)-V(s)^*V(s)-V(t)^*V(t)\\
&=& 2{\rm i}\, {\rm Im}L(s,t) + 2{\rm Re}L(s,t)-L(s,s)-L(t,t)\\
&=& 2{\rm i}\, {\rm Im}L(s,t)+ {\rm Re} K(s,t)-\frac{1}{2}(K(s,s)+K(t,t))\\
&=& {\rm i}\, {\rm Im}(K(s,t)-K(s,s_0)-K(s_0,t))+ {\rm Re} K(s,t)-\frac{1}{2}(K(s,s)+K(t,t))\\
&&\qquad\qquad \qquad\qquad\qquad\qquad \qquad\qquad (\mbox{since~} K(s_0,s_0) \mbox{ is self-adjoint})\\
&=& K(s,t)+ \left(\frac{-1}{2}K(s,s)-{\rm i}\, {\rm Im}K(s,s_0)\right)+ \left(\frac{-1}{2}K(t,t)-{\rm i}\, {\rm Im}K(t,s_0)\right)^*\\
&=& K(s,t)+h(s)+h(t)^*,
\end{eqnarray*}
where
\begin{eqnarray}\label{h}
h(s):=\frac{-1}{2}K(s,s)-{\rm i}\, {\rm Im}K(s,s_0).
\end{eqnarray}
Thus
$$K(s,t)=2V(s)^*V(t)-V(s)^*V(s)-V(t)^*V(t)-h(s)-h(t)^*.$$
\end{proof}

The triple $(V,\F, h)$ (or $(V,\F)$ when $h=0$, resp.) is called the minimal Kolmogorov decomposition of the conditionally positive definite kernel $K$.

The next result is related to the positive definiteness of functions of the form $\psi(s,t):=\varphi(s-t)$, where $\varphi$ is a real function on $\mathbb{R}^d$. It is a $C^*$-version of a known result in the Euclidean space $\mathbb{R}^d$; see \cite[Theorem A]{MIC} and references therein. It can be deduced from Theorem \ref{th2} but we provide a direct proof for it.

\begin{corollary}\label{cor2}
A matrix in $\mathbb{M}_n(\A)$ with self-adjoint entries and with zero diagonal entries is conditionally positive if and only if it is a sum of matrices of the form $[-|a_i-a_j|^2]$ with $a_1, \cdots, a_n \in \A$.
\end{corollary}
\begin{proof}
If $a_1, \cdots, a_n$ are entries of a $C^*$-algebra $\A$, then the matrix $[-|a_i-a_j|^2]$ is conditionally positive in $\mathbb{M}_n(\A)$ with self-adjoint elements and with zero diagonal entries. To prove this, let us use Corollary \ref{corm} and the fact that
$$-|a_i-a_j|^2=2\langle a_i,a_j\rangle - |a_i|^2-|a_j|^2.$$
Conversely, let $A=[a_{ij}]\in \mathbb{M}_n(\A)$ be a conditionally positive matrix with self-adjoint entries and with and with zero diagonal entries. Utilizing Corollary \ref{corm} we get a positive matrix $[b_{ij}]\in \mathbb{M}_n(\A)$ and self-adjoint elements of $c_1, \cdots, c_n \in \A$ such that $a_{ij}=b_{ij}+(c_i+c_j)$. Hence $[b_{i,j}]=\sum_{k=1}^m[{d_i^k}^*d_j^k]$ for some $d_i^k \in \A\,\,(1\leq i\leq n, 1 \leq k\leq m)$. From $a_{ii}=0\,\,(1\leq i\leq n)$ we conclude that $2c_i=-b_{ii}=-\sum_{k=1}^m |d_i^k|^2$. Since all $a_{ij}$'s and $c_i$'s are self-adjoint, we infer that so are all $b_{ij}$'s. Hence $\sum_{k=1}^m{d_i^k}^*d_j^k=b_{ij}=b_{ij}^*=b_{ji}=\sum_{k=1}^m{d_j^k}^*d_i^k$. Thus
$$a_{ij}=\sum_{k=1}^m{d_i^k}^*d_j^k-\frac{1}{2}\sum_{k=1}^m |d_i^k|^2-\frac{1}{2}\sum_{k=1}^m |d_j^k|^2=-\sum_{k=1}^m \left|\frac{d_i^k}{\sqrt{2}}-\frac{d_j^k}{\sqrt{2}}\right|^2.$$
\end{proof}

The problem of embedding of spaces endowed by types of metrics into Hilbert spaces goes back to the work of Schoenberg \cite{SCH}.  We aim to provide some conditions for embedding of a $C^*$-metric space into a Hilbert $C^*$-module. The next result has apparently an intrinsic relation to Corollary \ref{cor2}.

\begin{theorem}
Let $(S, d(\cdot,\cdot))$ be a $C^*$-metric space with values in a $C^*$-algebra $\A$. Then $S$ is $C^*$-isometric to a subset of a Hilbert $C^*$-module if and only if $K(s,t)=-d(s,t)^2$ is a conditionally positive definite kernel.
\end{theorem}
\begin{proof}
Let $V$ be a $C^*$-isometry from $S$ into a Hilbert $C^*$-module $(\E, \langle \cdot,\cdot\rangle)$ over a $C^*$-algebra $\A$. Put $L: S \times S \to \A$ by $L(s,t)=\langle V(s),V(t)\rangle$. Then $L$ is a positive definite kernel since for any $x_1, \cdots, x_n\in \E$ the Gram matrix $[\langle x_i,x_j\rangle]$ is positive; cf. \cite[Lemma 4.2]{LAN}. In addition,
\begin{eqnarray*}
K(s,t)&=&-d(s,t)^2=-|V(s)-V(t)|^2=-\langle V(s)-V(t), V(s)-V(t)\rangle\\
&=&-L(s,s)-L(t,t)+2{\rm Re}(L(s,t))=P(s,t)+Q(s,t),
\end{eqnarray*}
where $P(s,t):=-L(s,s)-L(t,t)$ is conditionally positive definite and $Q(s,t)=2{\rm Re}(L(s,t))=L(s,t)+L(t,s)$ is positive definite. Hence $K$ is clearly a conditionally positive definite kernel.

Conversely, let $K(s,t)=-d(s,t)^2$ be conditionally positive definite. Fixing $s_0 \in S$ and consider the positive definite kernel
$L(s,t):=\frac{1}{2}\big(K(s,t)-K(s,s_0)-K(s_0,t)+K(s_0,s_0)\big)$. Let us use the construction in Theorem \ref{th2} to get the Hilbert $C^*$-module $\mathcal{L}(\A,\F)$ and the map $V: S \to \mathcal{L}(\A,\F)$ (Note that in  Murphy's construction, we can start our work with a positive definite kernel with values in the set of `compact' operators $\mathcal{K}(\E)$ acting on a Hilbert $C^*$-module $\E$. In our case, we deal with the Hilbert $C^*$-module $\A$ over itself via the inner product $\langle a,b\rangle=a^*b$ and apply the fact that $\mathcal{K}(\A)$ is nothing than $\A$). Then we have
\begin{eqnarray*}
|V(s)-V(t)|^2&=&(V(s)-V(t))^*(V(s)-V(t))\\
&=&L(s,s)+L(t,t)-2{\rm Re}(L(s,t))\\
&=& -K(s,t)=d(s,t)^2.
\end{eqnarray*}
\end{proof}
Finally we study an order on the space of conditionally positive definite kernels. We say $K'\leq K$, where $K, K'$ are conditionally positive definite kernels whenever $K-K'$ is conditionally positive definite. The next theorem provides a characterization of conditionally positive definite kernels majorized by a given kernel $K$ under mild conditions. To get a suitable adjointable map, Pellonp\"a\"a \cite{PEL} considers a interesting regular condition. We say a kernel $K: S\times S \to \mathcal{L}(\E)$ is regular if the Hilbert $C^*$-module $\F$ of its minimal Kolmogorov decomposition $(V,\F)$ is self-dual. Then, by \cite[Proposition 2.5.2]{MT}, every bounded $A$-linear map on $\F$ is adjointable. In the case when $\A$ is finite dimensional (or equivalently, is a direct sum of matrix algebras $\mathbb{M}_m$), then every $\A$-module is self-dual.

\begin{theorem}\label{pro1}
Let $K, K'$ be regular conditionally positive definite kernels on a set $S$ into the $C^*$-algebra $\mathcal{L}(\E)$ for some Hilbert $C^*$-module $\E$ such that the minimal Kolmogorov decompositions $(V,\F)$ of $K$ is regular. Let there be $s_0\in S$ such that $K(s,s_0)$ and $K'(s,s_0)$ are self-adjoint for all $s\in S$ and both $K$ and $K'$ vanish on the diagonal $\{(s,s):s\in S\}$. Then $K'\leq K$ if and only if there exists a positive contraction $C\in \mathcal{L}(\F)$ such that $$K'(s,t)= 2V(s)^*C^*CV(t)-V(s)^*C^*CV(s)-V(t)^*C^*CV(t)$$
for all $s, t\in S$.
\end{theorem}
\begin{proof}
We use construction and notation in Theorem \ref{th2}. Let $(V,\F)$ and $(V',\F')$ be the minimal Kolmogorov decompositions of $K$ and $K'$, respectively. By the assumption, $K(s,s)=0$ for all $s \in S$ and there is $s_0\in S$ such that $K(s,s_0)$ is self-adjoint in $\A$ for all $s\in S$, so that ${\rm Im}K(s,s_0)=0$ for all $s\in S$. Hence $h(s)=0$ for all $s\in S$. A similar assertion holds about the function $h'$ corresponding to $K'$. Further, $K'\leq K$ if and only if $L' \leq L$ where $L$ and $L'$ are corresponding positive definite kernels to $K$ and $K'$ according to \eqref{l}, respectively.
By the Kolmogorov construction, $L' \leq L$ if and only if $\langle \L'f,\L'f\rangle \leq \langle \L f,\L f\rangle$ for all $f:S\to E$ with finite support. Now we show that this latter inequality is equivalent to $V'(s)=CV(s)$ for some positive contraction $C \in \mathcal{L}(\F)$:\\

Let $\langle \L'f,\L'f\rangle \leq \langle \L f,\L f\rangle$ for all $f:S\to E$ with finite support. Then the $\A$-linear map $W:\F_0\to \F'_0$ defined by $W(\L f)=\L'f$ is a contraction and can be extended to a contraction, denoted by the same $W$, from $\F$ to $\F'$. It follows from \cite[Proposition 2.5.2]{MT} that $W$ is self-adjoint. Thus for all $x\in \E$,
$$V'(s)(x)=\L'(x_s)=W(\L(x_s))=(WV(s))(x)\,,$$
whence
$V'(s)=WV(s)$. Thus
$$L'(s,t)=V'(s)^*V'(t)=(WV(s))^*(WV(t))=V(s)^*W^*WV(t)=V(s)^*CV(t),$$
where $C=W^*W$ and $\|C\|\leq \|W\|^2\leq 1$ since $W$ is a contraction.

Conversely, let $V'(s)=CV(s)$ for some positive contraction $C \in \mathcal{L}(\F)$. Then
\begin{eqnarray*}
\langle \L'f,\L'f\rangle&=&\sum_{s,t\in S}\langle K'(s,t)f(t),f(s)\rangle\\
&=&\sum_{s,t\in S}\langle V'(s)^*V'(t)f(t),f(s)\rangle\\
&=&\sum_{s,t\in S}\langle V(s)^*C^*CV(t)f(t),f(s)\rangle\\
&=&\langle C\sum_{t\in S}V(t)f(t),C\sum_{s\in S}V(s)f(s)\rangle\\
&=&\left\|C\sum_{s\in S}V(s)f(s)\right\|^2\\
&\leq& \left\|\sum_{s\in S}V(s)f(s)\right\|^2\\
&=&\sum_{s,t\in S}\langle K(s,t)f(t),f(s)\rangle\\
&=& \langle \L f,\L f\rangle
\end{eqnarray*}

Finally, by employing \eqref{v}, we observe that $K' \leq K$ if and only if
\begin{eqnarray*}
K'(s,t)&=&2V'(s)^*V'(t)-V'(s)^*V'(s)-V'(t)^*V'(t)\\
&=& 2V(s)^*C^*CV(t)-V(s)^*C^*CV(s)-V(t)^*C^*CV(t)\,.
\end{eqnarray*}
\end{proof}

\textbf{Acknowledgements.} The author would like to thank the referee for several useful comments improving the paper.

\end{document}